\documentclass[11pt]{amsart}
\usepackage {amsmath, amssymb,   epsfig, a4wide, enumerate, psfrag}
\usepackage[latin1]{inputenc}
\usepackage[english]{babel}

\date{\today}

\author{Romain Dujardin}
\thanks{Research   partially supported by ANR project BERKO}
\address{CMLS \\ \'Ecole Polytechnique \\ 91128 Palaiseau\\
         France}
\email{dujardin@math.polytechnique.fr}

\title{The supports of higher bifurcation currents}

\newcommand{\cc}{\mathbb{C}}

\newcommand{\dd}{\mathbb{D}}

\newcommand{\nn}{\mathbb{N}}

\newcommand{\cv}{\rightarrow}

\newcommand{\om}{\Omega}
\newcommand{\set}[1]{\left\{#1\right\}}
\newcommand{\norm}[1]{\left\Vert#1\right\Vert}
\newcommand{\abs}[1]{\left\vert#1\right\vert}

\newcommand{\pd}{{\mathbb{P}^2}}
\newcommand{\pu}{{\mathbb{P}^1}}
\newcommand{\pk}{{\mathbb{P}^k}}

\newcommand{\rest}[1]{ \arrowvert_{#1}}
\newcommand{\m}{{\bf M}}
\newcommand{\unsur}[1]{\frac{1}{#1}}

\newcommand{\lrpar}[1]{\left(#1\right)}
\newcommand{\bra}[1]{\left\langle #1\right\rangle}

\newcommand{\la}{\lambda}
\newcommand{\lo}{{\lambda_0}}

\newcommand{\La}{\Lambda}

\newcommand{\PSL}{\mathrm{PSL}(2,\mathbb C)}
\newcommand{\tbif}{T_{\mathrm{bif}}}

\newcommand{\hL}{\widehat{\Lambda}}
\newcommand{\hT}{\widehat{T}}
\newcommand{\hf}{\widehat{f}}

\newcommand{\itm}{\item[-]}

\DeclareMathOperator{\supp}{Supp}

\DeclareMathOperator{\codim}{codim}

\newtheorem{prop}{Proposition} [section]
\newtheorem{thm}[prop] {Theorem}

\newtheorem{lem}[prop] {Lemma}
\newtheorem{cor}[prop]{Corollary}

\theoremstyle{remark}

\newtheorem{rmk}[prop]{Remark}

\begin{document}

\begin{abstract}
Let $(f_\lambda)_{\lambda\in\Lambda}$ be a holomorphic family of rational mappings of degree $d$ on $\mathbb{P}^1(\mathbb{C})$, with $k$ marked critical points $c_1, \ldots , c_k$. To this data is associated a closed positive current $T_1\wedge \cdots \wedge T_k$ 
of bidegree $(k,k)$ on $\Lambda$, aiming to describe the simultaneous bifurcations of the marked critical points. In this note we show that the support of this current is accumulated by parameters at which  $c_1, \ldots , c_k$ eventually fall on repelling cycles. Together with results of Buff, Epstein and Gauthier, this leads to a complete characterization of $\mathrm{Supp}(T_1\wedge \cdots \wedge T_k)$. 
\end{abstract}

\maketitle

\section*{Introduction}

Let $(f_\la)_{\la\in\La}$ be a holomorphic family of rational mappings of degree $d$ on $\pu(\cc)$, parameterized by a complex manifold $\La$. Let $c = (c(\la))_{\la\in \La}$ be a marked (i.e. holomorphically moving) critical point.  It is an immediate consequence of Montel's theorem that if $c$ bifurcates at $\lo$, then there exist nearby parameters where the orbit of 
$c$ falls eventually on a repelling cycle. Now if $k>1$ critical points $c_1, \ldots , c_k$ are marked, and simultaneously bifurcate at $\lo$, it is natural to wonder whether they can be perturbed to be made simultaneously preperiodic. An example due to  Douady  shows that this is impossible in general (see 
\cite[Example 6.13]{preper}).

It appears that the right language to deal with this type of questions is that of  
  bifurcation currents, which we briefly review now. The reader is referred to  \cite{survey, berteloot survey} for a thorough presentation. To a marked critical point $c$, following \cite{preper}, one can associate a bifurcation current $T_c$ on $\La$, whose support coincides with the activity (bifurcation) locus of $c$. When $k$ critical points $c_1, \ldots , c_k$ are marked, the wedge product $T_1\wedge \cdots \wedge T_k$ is well defined and its support is contained (but not equal to) the locus where the $c_i$ are  simultaneously active. 

Assume that for $\lo\in \La$, the marked  critical points $c_j$ eventually land on repelling periodic points $p_j$, that is $f^{n_j}_\lo(c_j(\lo)) = p_j(\lo)$. We let the mappings $C$ and $P :\La\cv(\pu)^k$ be respectively  defined by $C: \la\mapsto 
(f^{n_1}_\la(c_1(\la)), \ldots, f^{n_k}_\la(c_k(\la)))$ and $P: \la\mapsto (p_1(\la), \ldots , p_k(\la))$, where  $p_j(\la)$ is the natural continuation of $p_j(\lo)$ as a periodic point. We say that the $c_j$ {\em fall  transversely} onto the periodic points  $p_j$ if the graphs of the 
 mappings $P$ and $C$ are transverse at $\lo$. 

 Our main result  is the following.

\begin{thm}\label{thm:main}
Let $(f_\la)_{\la\in \La}$ be a holomorphic family of rational maps of degree $d\geq 2$ on $\pu$. 
Assume that  $c_1,\ldots,  c_k$  are  marked critical points and  let $T_1,\ldots, T_k$ be the respective 
bifurcation currents. 

Then  every parameter in $\supp(T_1\wedge \cdots \wedge T_k)$ is accumulated by parameters $\la$  
for which  $c_1(\la),\ldots, c_k(\la)$ fall transversely  onto repelling cycles.  
\end{thm}

Conversely, Buff and Epstein \cite{buff epstein} showed that if at $\lo$, $c_1(\lo),\ldots, c_k(\lo)$ fall transversely  onto repelling cycles, then $\lo\in \supp(T_1\wedge \cdots \wedge T_k)$.
Altogether this leads to the following precise characterization of 
$\supp(T_1\wedge \cdots \wedge T_k)$.
 
\begin{cor}\label{cor:support}
Under the assumptions of Theorem \ref{thm:main},  $\supp(T_1\wedge \cdots \wedge T_k)$ is the 
closure of the set of  parameters $\la$  for which  $c_1(\la),\ldots, c_k(\la)$ fall transversely
  onto repelling  cycles. 
\end{cor}

By using a result of Gauthier \cite{gauthier} we may replace the word ``transversely"   by ``properly" in the above corollary (see \S \ref{sec:polynomial} for details). 

Besides the case $k=1$, this result was previously only known when  $\La$ is the space of all polynomials or rational maps of degree $d$ with marked critical points and $k$ is maximal \cite{preper, buff epstein, buff gauthier}. 

There is  also an ``absolute" bifurcation current $\tbif$, which was introduced, prior to \cite{preper}, by DeMarco \cite{demarco}. When all critical points are marked, $\tbif$ is just the sum of the associated bifurcation currents.  It was shown by Bassanelli and Berteloot that $\supp(\tbif^k)$ is contained in the closure of the set of parameters possessing $k$ distinct neutral (resp. attracting or superattracting) orbits (see \cite{basber1, basber2, berteloot survey}). It follows from Corollary \ref{cor:support} that  $\supp(\tbif^k)$ is the closure of the set of parameters where at least $k$ critical points fall transversely onto repelling cycles. 

\medskip

The arguments required for the proof of Theorem \ref{thm:main}, along
 with some  classical techniques from
  value distribution theory   also lead
 to the following interesting characterization  of $T_1\wedge \cdots \wedge T_k$, in the spirit of higher dimensional holomorphic dynamics.  

\begin{thm}\label{thm:value}
Let $(f_\la)_{\la\in \La}$ be a holomorphic family of rational maps of degree $d\geq 2$ on $\pu$. 
Assume that  $c_1,\ldots,  c_k$  are  marked critical points and  let $T_1,\ldots, T_k$ be the respective 
bifurcation currents. 

There exists a pluripolar set $\mathcal{E}\subset (\pu)^k$ such that if $(z_1, \ldots, z_k)\in \pu^k
\setminus \mathcal{E}$, the following equidistribution statement holds:  
$$\unsur{d^{nk}}\left[ \set{f_\la^n(c_1(\la)) = z_1}\cap \cdots \cap \set{f_\la^n(c_k(\la)) = z_k}\right]\cv 
T_1\wedge \cdots  \wedge T_k.$$
\end{thm}

\medskip

 The plan of the paper is the following. After some preliminaries in \S\ref{sec:prel}, we prove Theorem 
 \ref{thm:value} in \S\ref{sec:value}. In contrast with \cite{buff epstein}, the transversality assertion in Theorem \ref{thm:main} does not follow from dynamical considerations. It relies on a  
  a transversality result of 
 independent interest,  which is established in \S\ref{sec:transversality}.  The theorem itself is proved in 
 \S\ref{sec:main}. Finally in \S\ref{sec:polynomial}, we specialize to the case where $\La$ is the space
  of polynomial or rational maps of degree $d$, and show that  
   Corollary \ref{cor:support} can be made more precise in this case.
  
 \medskip
  
 \noindent{\bf Acknowledgments.} Thanks to Thomas Gauthier for  many useful comments.  
 
\section{Preliminaries on bifurcation currents} \label{sec:prel}

We first  briefly recall from \cite{preper} the construction of the bifurcation current associated to a 
critically marked family, as well  as a few necessary estimates. 

\medskip

We let $(f_\la)_{\la\in \La}$ be a family of rational mappings of degree $d\geq 2$, with a marked critical 
point $c:\la\mapsto c(\la)$. Let $\hf$ be the product mapping defined on $\La\times \pu$ by $\hf(\la, z) = (\la, f_\la(z))$. Fix a 
Fubini-Study form $\omega$ on $\pu$. 
We let   $\pi_\La$ and $\pi_{\pu}$ be the coordinate projections on $\La\times\pu$
 and   $\widehat \omega = \pi_{\pu}^*\omega$. 

It is not difficult to show that the sequence $d^{-n}(\widehat{f}^n) ^*\widehat\omega$ converges to a 
current $\widehat T$. More precisely we have that 
\begin{equation}\label{eq:gn}
 d^{-n}(\widehat{f}^n)^*\widehat\omega = \widehat{\omega} + dd^c g_n \cv \omega+ dd^c g_\infty = 
\widehat T
\text{, with }\norm{g_n-g_\infty}_{L^\infty} = O(d^{-n})
\end{equation}
  where the constant in the $O(\cdot)$ is locally uniform on $\La$. Let $\Gamma_c$ be the graph of $c
$ in $\hL$.
  The {\em bifurcation current associated to $c$} is by definition $T=  (\pi_\La)_*\lrpar{\widehat{T}
\rest{\Gamma_{c}}}$. Notice   that  $T$ has continuous potentials: to be specific, a local 
potential for $T$ is 
given by  $\la\mapsto \widehat g (\la, c(\la))$, where $\widehat g$ is a local potential of $\hT$. 

If we let $c_n$ be defined by $c_n(\la) = f^n_\la(c(\la))$, then 
$d^{-n}c_n^*\omega = (\pi_\La)_* \lrpar{d^{-n} \widehat\omega\rest{\hf^n(\Gamma_c)}}$, so 
$d^{-n}c_n^*\omega \underset{n\cv\infty}\longrightarrow T$ and the difference between the potentials 
is $O(d^{-n})$, locally uniformly on $\La$. 

\medskip

The following result will be useful. It was first obtained in this general form by Gauthier \cite[Thm 6.1]{gauthier}, with a different proof. 

\begin{lem}\label{lem:wedge}
Let as above $((f_\la)_{\la\in\La}, c)$ be a critically marked family, and $T$ be the associated 
bifurcation current. 
Then $T\wedge T =0$.  
\end{lem}

\begin{proof}
Uniform convergence of the potentials implies that 
$$ d^{-n}{c}_n^{  *}  \omega \wedge d^{-n} {c}_n^{  *}  \omega
= d^{-2n}{c}_n^{  *}( \omega^2) \underset{n\cv\infty}\longrightarrow  
 T^2.$$ 
But
  $\omega^2 =0$ so the result follows.
\end{proof}

\section{Value distribution of post-critical points}\label{sec:value}

 This section is dedicated to Theorem \ref{thm:value}. Its proof mimicks that of the equidistribution of 
preimages for endomorphisms of $\pk$ so we follow   Sibony \cite[Theorem 3.6.1]{sibony},
with some extra care  due to the fact that $\La$ is not compact. For clarity we use the formalism of 
intrinsic capacities  of Guedj and Zeriahi \cite{gz1}.
 Another option would be to adapt the arguments of \cite{ds-tm}. 
  
 
 \medskip
 
Before starting the proof, a few remarks  on the definitions of pullback currents and measures are in 
order. It is difficult (if not impossible) to give a reasonable definition of the pullback of a general 
positive current of bidegree $(p,p)$ under a holomorphic map. The main trouble is the behavior of pullback 
currents under weak limits. 
On the other hand, there are favorable situations where the definition works well:
\begin{itemize}
\itm when $T$ is a positive closed current of bidegree $(1,1)$, in which case one can pullback the 
potentials;
\itm when the holomorphic map is a submersion.  
\end{itemize}

In our setting, we will consider a sequence of dominant holomorphic mappings $C_n:\La\cv (\pu)^k$, 
and need to pullback probability measures (i.e. positive currents of bidegree $(k,k)$  on $(\pu)^k$.  
We  will consider probability measures $\nu$ which are products of positive closed currents with 
bounded potentials $\nu = S^k$. Thus, 
$\nu$ gives    no mass to the set $\mathrm{CV}(C_n)$ 
of critical values of $C_n$. Likewise, $(C_n^*S)^k$  is a product of currents with bounded  local 
potentials so it is of locally bounded mass near 
 $C_n^{-1}(\mathrm{CV}(C_n))$ and gives no mass to it. 
Therefore, we can  define $C_n^*\nu$ to be the extension 
of $C_n^*\lrpar{\nu\rest{(\pu)^k\setminus \mathrm{CV}(C_n)}}$ by zero on $C_n^{-1}(\mathrm{CV}
(C_n))$, and we have that 
$C_n^*\nu  = (C_n^*S)^k$. 

Pulling back $\nu$ under $C_n$ is also compatible with the disintegration 
  $\nu = \int \delta_a d\nu(a)$, that is,  $C_n^*\nu = \int C_n^*\delta_a d\nu(a)$, since again it is enough 
to restrict to the values of $a$ lying outside  $\mathrm{CV}(C_n)$. 

Notice that we will never have to consider   weak limits on the target space $(\pu)^k$. 

\begin{proof}[Proof of Theorem  \ref{thm:value}] 
 By assumption, $(f_\la)_{\la\in \La}$ is a holomorphic family of rational mappings of degree $d\geq 2$ 
with $k$ marked critical points $c_1,\ldots ,c_k$. It is no loss of generality to assume that $\La$ is a ball 
in $\cc^{\dim(\La)}$.  We let $C_n:\La \cv (\pu)^k$ be defined by 
 $$C_n(\la) = (f^n_\la(c_1(\la)), \ldots , f^n_\la(c_k(\la))) = (c_{n, 1}(\la), \ldots, c_{n,k}(\la)).$$ On $
(\pu)^k$ consider the Kähler form $\kappa =  {(k!)^{-1/k}}(p_1^*\omega + \cdots + p_k^*\omega)$, 
where $p_j=(\pu)^k\cv\pu$ is the projection on the $j^{\rm th}$ factor. 
Then $ p_1^*\omega\wedge  \cdots \wedge p_k^*\omega =  \kappa^k$, and uniform convergence of 
the potentials implies that  $C_n^*\lrpar{ \kappa^k}$ converges to 
$T_1\wedge \cdots \wedge T_k$ as $n\cv\infty$. 

Our purpose is to show that the set of $a\in (\pu)^k$ such that $d^{-nk}C_n^*\delta_a$ does not 
converge to $T_1\wedge \cdots \wedge T_k$ is pluripolar. Fix a test form $\psi$ of bidegree $(\dim
(\La)-k, \dim(\La)-k)$ on $\La$. 
For fixed $n$ and $s>0$, we   estimate the Monge-Ampère capacity of 
$$E_{n,s}^+: = \set{a, \ {\unsur{d^{nk}}} \bra{C^*_n \delta_a, \psi} - \unsur{d^{nk}} \langle C_n^*
\kappa^k, \psi\rangle  \geq s}.$$
   If $E_{n,s}^+$ is pluripolar the capacity is zero so  there is nothing to prove. 
 This happens in particular when $C_n$ is not dominant, since in this case $\langle C_n^*\kappa^k, \psi
\rangle=0$ and $\bra{C^*_n \delta_a, \psi}=0$ for $a$ lying outside an analytic set. Hence in what follows we 
can assume that $C_n$ is dominant.
Consider a 
non-pluripolar compact subset $E\subset E_{n,s}^+$ and  introduce the so-called Siciak extremal 
function  \cite[\S 5.1]{gz1} 
$$v_{E, \kappa} = \sup\set{u \ \kappa\text{-psh}, \ u\leq 0 \text{ on } E}.$$ Then the upper 
semi-continuous regularization $v:=v_{E, \kappa}^*$ is $\kappa$-psh, non-negative, and 
$$\int_E \kappa_v^k = \int_{(\pu)^k} \kappa_v^k = \int_{(\pu)^k} \kappa^k =1,$$ where $\kappa_v = 
\kappa + dd^cv$.
Following \cite{sibony} we compute  
\begin{align}
s&\leq \int_E \lrpar{\unsur{d^{nk}} \bra{C^*_n \delta_a, \psi} - \unsur{d^{nk}} \langle C_n^*\kappa^k, \psi
\rangle }\kappa_v^k(a)\notag \\
&
= \unsur{d^{nk}} \int_\La \psi \wedge \lrpar{ C_n^* \kappa_v^k  - C_n^*\kappa^k} \notag\\
& = \int_\La \psi  \wedge \frac{dd^c (v\circ C_n)}{d^n} \wedge \lrpar{\unsur{d^{n(k-1)}}\sum_{j=0}^{k-1} 
C_n^*\kappa_v^{k-1-j}\wedge C_n^*\kappa^j} \label{eq:notag}.
\end{align}
Let us  estimate the mass (denoted by $\m(\cdot)$) of the current between brackets in the last integral. 
Let $u_j$ be a potential of $c_j^*\omega$ on $\La$. Then   a potential of $d^{-n}C_n^*\kappa$  is 
defined by the formula
\begin{equation}\label{eq:kappa}
 \unsur{(k!)^{1/k}}\sum_{j=1}^k u_j + dd^c(g_{n, j})
\end{equation}
where $g_{n,j}$ is as in \eqref{eq:gn}. Observe that this sequence of psh functions is uniformly bounded
by some constant $M$. Likewise, to obtain a potential of $d^{-n}C_n^*\kappa_v$, it is 
enough to add $d^{-n}v\circ C_n$ to \eqref{eq:kappa}. It then follows from the Chern-Levine-Nirenberg 
inequality that
\begin{align*}
\m \lrpar{\unsur{d^{n(k-1)}}\sum_{j=0}^{k-1} 
C_n^*\kappa_v^{k-1-j}\wedge C_n^*\kappa^j} 
&\leq C^{  st} \sum_{j=0}^{k-1} M^j \lrpar{ M+ \frac{\norm{v}_{L^\infty}}{d^n}}^{k-1-j} \\
&\leq C^{  st} \sum_
{j=0}^{k-1}  \lrpar{ 1+ \frac{\norm{v}_{L^\infty}}{d^n}}^j.
\end{align*}
Plugging this into \eqref{eq:notag} and integrating by parts we infer that 
$$s\leq  {C^{  st}}  \norm{\psi}_{C^2}\frac{\norm{v}_{L^\infty}}{d^n}\sum_{j=0}^{k-1}  \lrpar{ 1+ \frac
{\norm{v}_{L^\infty}}{d^n}}^j.$$ Finally, from  elementary calculus we conclude that 
$$\frac{\norm{v}_{L^\infty} }{d^n}  \geq \lrpar{\frac{ s}{C^{\rm st}\norm{\psi}_{C^2}}+1}^{1/k} - 1=:  h(s), 
\text{ that is, }  \norm{v}_{L^\infty} \geq d^n  h(s).
$$

The   Alexander capacity of $E$ (relative to $\kappa$) is defined  by  $$T_\kappa(E) = \exp(-\sup_X v) = 
\exp(- \norm{v}_{L^\infty}) $$  (see \cite[\S 5.2]{gz1}).
Thus by the above inequality we get that $T_\kappa(E) \leq \exp \lrpar{-d^n h(s) / C\norm{\psi}_{C^2}}
$, and since $E\subset E_{n,s}$ is an arbitrary compact subset, the same inequality holds for  
   $T_\kappa(E_{n,s}^+)$. We do not need to define the Monge-Ampère capacity $\mathrm{cap}_
\kappa$ precisely here, but only recall that it is a subadditive capacity on $(\pu)^k$ which vanishes 
precisely on pluripolar sets and satisfies an inequality of the form 
   $\mathrm{cap}_\kappa(E)\leq \frac{-A }{\log T_\kappa(E)}$ when $T_\kappa(E)$ is small \cite[\S 7.1]{gz1}. So we infer 
that 
    $$\mathrm{cap}_\kappa(E_{n,s}^+)\leq \frac{A}{d^nh(s)}. $$    
Reversing   inequalities in the definition of $E^+_{n,s}$, we obtain a similar estimate for 
$$E_{n,s}: = \set{a, \   \abs{\unsur{d^{nk}} \bra{C^*_n \delta_a, \psi} - \unsur{d^{nk}} \bra{C_n^*
\kappa^k, \psi} }\geq s}.$$
 By subadditivity of  $\mathrm{cap}_\kappa$, we infer that 
 for every $s>0$, $\mathrm{cap}_\kappa \lrpar{\bigcap_{n_0}\bigcup_{n\geq n_0}   E_{n,s}} =0$, so  
 we conclude that for $a$ outside a pluripolar set, 
   $${\unsur{d^{nk}}} \bra{C^*_n \delta_a, \psi}\underset{n\cv\infty}\longrightarrow  \bra{T_1\wedge 
\cdots \wedge T_k, \psi}.  $$
To complete the proof it is enough to  consider  a countable dense family of test forms $\psi$.
\end{proof}

Under an additional  global assumption on the family $(f_\la)$, the following result follows directly from 
Theorem \ref{thm:value}. 

\begin{cor}\label{cor:non pluripolar}
Let $(f_\la)_{\la\in \La}$ be an {\em algebraic} family of rational maps of degree $d\geq 2$ on $\pu$. 
Assume that  $c_1,\ldots,  c_k$  are  marked critical points and  let $T_1,\ldots, T_k$ be the respective 
bifurcation currents. Let $C_n$ be defined by  $C_n(\la) = (f^n_\la(c_1(\la)), \ldots , f^n_\la(c_k(\la))) = 
(c_{n, 1}(\la), \ldots, c_{n,k}(\la))$. 

Then if $\nu$ is a measure on $(\pu)^k$ which gives no mass to pluripolar sets, the sequence of 
currents $d^{-nk} C_n^*\nu$ converges to $T_1\wedge \cdots \wedge T_k$.  
\end{cor}

Here, algebraic means that $\La$ is an open subset of a  quasi-projective variety and that $(f_\la)$ 
depends algebraically on $\la$.
It is unclear whether  this algebraicity assumption is   really necessary.  

\begin{proof}
For $\nu$-a.e. $a$, we have that $d^{-nk} C_n^*\delta_a \cv T_1\wedge \cdots \wedge T_k$. The point 
is to be able to integrate with respect to $a$. Passing to a quasi-projective variety containing $\La$ if 
necessary, it is no loss of generality to assume that $\La$ is quasi-projective. By assumption 
   $C_n:\La\cv(\pu)^k$ is a sequence of rational mappings.  Therefore 
 for fixed  $n$, the mass (degree) of   $d^{-nk}C_n^*\delta_a$ is independent of $a$ on some Zariski 
open subset of $(\pu)^k$. On the other hand, the average value of this mass is equal to that of 
 $d^{-nk}C_n^*\kappa^k$ which is bounded in $n$.  The result then follows from the dominated convergence theorem.
\end{proof}

%
%

\section{A transversality theorem} \label{sec:transversality}

As a preliminary step for the proof of Theorem \ref{thm:main}, in this section we discuss the intersection of uniformly laminar  currents of bidegree (1,1) in higher dimension. Beyond the fact that the wedge product of uniformly laminar currents admits (as expected) a geometric interpretation,  
our main purpose is to show that  the intersections between the  leaves are generically transverse. 
The two dimensional case was treated in \cite{bls} (see also \cite{isect}). 

\medskip

We first recall some basics on the intersection theory of holomorphic chains in an open set $\om\subset\cc^d$. 
See Chirka \cite[Chap. 12]{chirka} for a fuller account on this. 
Recall that a \emph{holomorphic chain} in $\om$
is a formal combination $Z = \sum k_iA_i$ of distinct irreducible analytic subsets of $\om$ with integer coefficients  (the \emph{components }of $Z$).
 The number $k_i$ is the \emph{multiplicity} of the component $A_i$. We'll only have to 
consider positive chains of pure dimension,
 that is, for which the multiplicities are positive and the components have the same dimension. The multiplicity of a chain at a point is the sum of that of its components; the \emph{support} of a chain is the union of its components.  When there is no risk of confusion we sometimes identify a chain and its support. 

A family of holomorphic chains $Z_1, \ldots , Z_k$ in $\om$ is said to \textit{intersect properly}
at $p$ (resp. in $\om$)  if at $p$ (resp. for every $p\in \om$) the intersection of the supports is of minimal possible dimension, that is,
 $\codim_p(\bigcap_{i=1}^k\supp(Z_i)) =  \sum_{i=1}^k \codim_p(\supp(Z_i)).$ 
If the intersection of $Z_1, \ldots , Z_k$ is proper at  $p$, there is a well defined intersection 
index $i_p(Z_1, \ldots , Z_k)$, which is constant on the regular part of every irreducible component of
$\supp(Z_1\cap\cdots \cap Z_k)$. 
If this intersection is proper in $\om$, then we may define the intersection chain 
$Z_1\wedge \cdots \wedge Z_k$, whose multiplicity along a given  irreducible component of
$\supp(Z_1\cap\cdots \cap Z_k)$ is the generic value of  $i_p(Z_1, \ldots , Z_k)$. 
 
If this intersection is not proper, by definition we put $Z_1\wedge \cdots \wedge Z_k=0$.

 If $Z_1, \ldots , Z_k$ are smooth and intersect properly at a given point $p$, then the intersection is transverse at $p$ if and only if $i_p(Z_1, \ldots , Z_k)$ is the product of the multiplicities of the $Z_i$ at $p$. In particular if 
 $A_1, \ldots , A_k$ are submanifolds of $M$, and $C$ is a component   of $A_1\wedge \cdots \wedge A_k$, then $C$ is of multiplicity 1 if and only if 
 the  $A_i$ intersect transversely along the regular part of $C$. 

\medskip

Recall that a current $A$ of bidegree (1,1) in $\om\subset \cc^d$ is said to be  {\em uniformly laminar}
if there exists a lamination by hypersurfaces in $\om$ such that the restriction to any flow box is of the form 
$\int [A_\alpha] da(\alpha)$, where $a$ is a positive measure on a transversal and the $A_\alpha$ are the plaques.  
Viewing $A_\alpha$ as a  graph over $\dd^{d-1}$ in $\dd^d$, we can write $A_\alpha = 
\set{(z,w), \ w = h_\alpha(z)}$, where $h_\alpha$ is a bounded holomorphic function on $\dd^{d-1}$. 

Putting $u_\alpha(z,w) = \log\abs{w- h_\alpha(z)}$, we get that $[A_\alpha] = dd^c u_\alpha$, and the family of psh functions $u_\alpha$ is locally uniformly bounded in $L^1_{\rm loc}$ .
Therefore $ u= \int u_\alpha da(\alpha)$ is a well-defined psh function and $A = dd^cu$.

\medskip

Our main result in this section is the following. The result is local so we can work with uniformly 
laminar currents restricted to  flow boxes. 

\begin{thm}\label{thm:isect geom}
For $i=1,\ldots, k$, let $A_i = \int[A_{i, \alpha_i}] da_i(\alpha_i)$ be uniformly laminar currents with bounded potentials in some open set $\om\subset \cc^d$, and let $S$ be an irreducible analytic subset of $\om$. Then the intersection $A_1\wedge \cdots \wedge A_k\wedge [S]$ is geometric in the sense that
\begin{equation}\label{eq:isect geom}
 A_1\wedge \cdots \wedge A_k\wedge [S] = \int [A_{1, \alpha_1}\wedge \cdots \wedge A_{k, \alpha_k}\wedge S] da_1(\alpha_1)\cdots da_k(\alpha_k). 
\end{equation}

In particular  only proper intersections account for the wedge product $A_1\wedge \cdots \wedge A_k\wedge [S]$. 
Furthermore for a.e. $(\alpha_1, \cdots , \alpha_k)$, 
 the   intersection $A_{1, \alpha_1}\wedge \cdots \wedge A_{k, \alpha_k}\wedge S$, if non-trivial, is a chain of multiplicity 1. In particular the submanifolds
$A_{1, \alpha_1}, \ldots ,  A_{k, \alpha_k}, S$ are transverse along the regular part of their intersection.  
\end{thm}

\begin{proof}
Since $A_1\wedge \cdots \wedge A_k\wedge [S]$ carries no mass on $\mathrm{Sing (S)}$ we may suppose that $S$ is smooth.
We argue by induction on $k$. 
Assume that for some $0\leq q<k$ the result holds for $A_1\wedge \cdots \wedge A_q\wedge [S]$ (of course of $q=0$ this expression   simply means $[S]$, and the result is true in this case).
Fix $\alpha_1, \cdots , \alpha_q$ and let $V =  \supp(A_{1, \alpha_1}\wedge \cdots \wedge A_{q, \alpha_q}\wedge S)$.  Write 
$A_{q+1, \alpha_{q+1}} = \set{w = h_{q+1, \alpha_{q+1}}(z)}$, 
with $\abs{h_{q+1, \alpha_{q+1}}}<1$, in some local system of coordinates. 
As above, 
the function  $u_{q+1}  :=  \int \log\abs{w-h_{q,\alpha_q}(z)} da_q(\alpha_q)$ is a potential of $A_q$, hence bounded by assumption, therefore
    $u_{q+1}  :=  \int \log\abs{h_{q+1,\alpha_{q+1}}} da_{q+1}(\alpha_{q+1})$ is locally integrable on $V$. Since the logarithms are negative,  by Fubini's theorem we infer that  for  a.e. $\alpha_{q+1}$, $\log\abs{h_{q+1,\alpha_{q+1}}} \in L^1_{\rm loc}(V)$, so $V\not\subset A_{q+1,\alpha_{q+1}}$. In particular for  a.e. $\alpha_{q+1}$, 
the intersection $V\cap A_{q+1,\alpha_{q+1}}$ is proper. 

\medskip

The justification of the formula \eqref{eq:isect geom} is classical.  We rely on  the following two facts:
\begin{itemize}
\itm If $u = \int u_\alpha da(\alpha)$ is an integral of negative  psh functions and $T$ is a positive current such that $uT$  (resp $u_\alpha T$  for a.e.   $\alpha$)  has finite mass, then $uT = \int u_\alpha Tda(\alpha)$.
\itm Likewise if $T = \int T_\alpha da(\alpha)$ is an integral of positive currents 
and $u$ is a negative psh function such that  $uT$ (resp. $uT_\alpha$ for a.e. $\alpha$) 
has finite mass, then $uT = \int u  T_\alpha da(\alpha).$
\end{itemize}

 Assume by induction  that \eqref{eq:isect geom}
  holds for $k=q$, and   write
\begin{align}
  A_1\wedge \cdots \wedge & A_q\wedge A_{q+1}\wedge [S]  = dd^c \lrpar{ u_{q+1} A_1\wedge \cdots \wedge A_q\wedge [S]} \label{eq:align}\\
&= dd^c  \lrpar{u_{q+1}\int [A_{1, \alpha_1}\wedge \cdots \wedge A_{q, \alpha_q}\wedge S] da_1(\alpha_1)\cdots da_q(\alpha_q)} \notag \\ 
&=  dd^c\lrpar{ \int u_{q+1} [A_{1, \alpha_1}\wedge \cdots \wedge A_{q, \alpha_q}\wedge S] da_1(\alpha_1)\cdots da_q(\alpha_q)} \notag \\ 
&=\int \lrpar{ (dd^cu_{q+1})  \wedge [A_{1, \alpha_1}\wedge \cdots \wedge A_{q, \alpha_q}\wedge S]} da_1(\alpha_1)\cdots da_q(\alpha_q). \notag
\end{align}
Now, as before, fix $\alpha_1, \cdots , \alpha_q$ and put $V = A_{1, \alpha_1}\wedge \cdots \wedge A_{q, \alpha_q}\wedge S$ (here $V$ is the chain, not   the mere support), and consider 
$$   u_{q+1}    [V] = \lrpar{\int \log\abs{h_{q+1,\alpha_{q+1}}} da_{q+1}(\alpha_{q+1}) }  [V].  
$$
Let $\mathcal{A}_R$ be the set of values $\alpha_{q+1}$ so that 
$\norm{\log\abs{h_{q+1,\alpha_{q+1}}}}_{L^1(V)} \leq R$, and let 
 $$u_{q+1, R} =   \int_{\mathcal{A}_R}
 \log\abs{h_{q+1,\alpha_{q+1}}} da_{q+1}(\alpha_{q+1}) .$$ Notice that  $u_{q+1, R}$ decreases to $u_{q+1}$ as $R$ increases to infinity.  We then infer that 
 $$u_{q+1, R}[V] = \int_{ \mathcal{A}_R}
 \log\abs{h_{q+1,\alpha_{q+1}}}[V]  da_{q+1}(\alpha_{q+1}).$$ Thus by   monotone convergence   we obtain  that  $u_{q+1}    [V] = \int 
 \log\abs{h_{q+1,\alpha_{q+1}}}[V]  da_{q+1}(\alpha_{q+1}).$ Finally, by taking the $dd^c$, we conclude that  
 $$(dd^c u_{q+1})\wedge [V] = \int [A_{q+1,\alpha_{q+1}}]\wedge [V] da_{q+1}(\alpha_{q+1})=
 \int [ A_{q+1,\alpha_{q+1}}\wedge V] da_{q+1}(\alpha_{q+1}),$$ where the second equality comes from the fact that intersection in the sense of currents of properly intersecting analytic sets
  coincides with 
 intersection in the sense of chains \cite[\S 16.2]{chirka}. This, together with \eqref{eq:align}, concludes the proof of \eqref{eq:isect geom}.
 
 \medskip
 
 We now show that intersections are generically transverse. Again we argue by induction, so let us  assume that  for    $q<k$ the result holds for $A_1\wedge \cdots \wedge A_q\wedge [S]$. Then 
 for a.e. $\alpha_1,\ldots , \alpha_q$, every component of $A_{1, \alpha_1} \wedge \cdots \wedge A_{q, \alpha_q}\wedge S$ is of multiplicity 1.  Since $A_{q+1}$ has bounded potential, the intersection 
 $[A_{1, \alpha_1} \wedge \cdots \wedge A_{q, \alpha_q}\wedge S]\wedge A_{q+1}$ gives no mass to the singular locus of $A_{1, \alpha_1} \wedge \cdots \wedge A_{q, \alpha_q}\wedge S$. Therefore in the geometric intersection 
 $$A_1\wedge \cdots \wedge A_{q+1}\wedge [S] = \int [A_{1, \alpha_1} \wedge \cdots \wedge A_{q, \alpha_q} \wedge A_{q+1, \alpha_{q+1}}\wedge S] da_1(\alpha_1)\cdots da_{q+1}(\alpha_{q+1})$$
 we can restrict to those intersections lying in the regular part of $A_{1, \alpha_1} \wedge \cdots \wedge A_{q, \alpha_q}\wedge S$. 
 
 So let us consider  a proper intersection between a smooth part of $A_{1, \alpha_1} \wedge \cdots \wedge A_{q, \alpha_q}\wedge S$ (which will be fixed from now on) and some leaf 
 $A_{q+1, \alpha_{q+1}}$ of $A_{q+1}$. 
   Let $W$ be a component of this intersection (hence of codimension $q+1+\codim(S)$). We want to show that for generic $\alpha_{q+1}$, $W$ has multiplicity 1. Let $p$ belonging to the regular part of $W$. 
   Changing local coordinates, it is no loss of generality to assume that in the neighborhood of $p$, 
 $A_{1, \alpha_1} \wedge \cdots \wedge A_{q, \alpha_q}\wedge S$ is an affine subspace of codimension $q+\codim(S)$. 
   If   $A_{1, \alpha_1} \wedge \cdots \wedge A_{q, \alpha_q}\wedge S$ and $A_{q+1, \alpha_{q+1}}$  are transverse at $p$ there is nothing to prove. Otherwise $ A_{1, \alpha_1} \wedge \cdots \wedge A_{q, \alpha_q}\wedge S$ is contained in  the tangent space $T_p(A_{q+1, \alpha_{q+1}})$.
      Let $\Pi$ be the 2-dimensional linear subspace through $p$, generated by $(e,f)$, where $e$ is  a vector transverse to $W$ in $A_{1, \alpha_1} \wedge \cdots \wedge A_{q, \alpha_q}\wedge S$, 
 and $f$ is transverse to $T_p(A_{q+1, \alpha_{q+1}})$.  Then 
 $\Pi\cap A_{1, \alpha_1} \wedge \cdots \wedge A_{q, \alpha_q}\wedge S$ is a smooth curve: in our coordinates it is the line $L$ tangent to $e$ through $p$.  Note also  that $W\cap \Pi = \set{p}$.
 By transversality, $A_{q+1}$ induces a lamination by curves on $\Pi$ near $p$, which is tangent to $e$ at $\Pi$, and the leaf through $p$ has an isolated intersection with $L$.  In this 2-dimensional situation,  \cite[Lemma 6.4]{bls} implies that for {\em every} 
$\alpha$  close to $\alpha_{q+1}$, $\Pi\cap  A_{q+1,\alpha}$ intersects $L$ transversely 
  in $\Pi$, hence $A_{1, \alpha_1} \wedge \cdots \wedge A_{q, \alpha_q}\wedge S$ is transverse to $A_{q+1,\alpha}$ in $\om$. This finishes the proof. 
 \end{proof}

\begin{rmk}
The argument  was written so that the proof of \eqref{eq:isect geom}, as well as generic properness of intersections, carries over without modification to the case where the $A_i$ are  \emph{uniformly woven} currents, that is, when the leaves are allowed to intersect. On the other hand generic transversality needn't hold in this more general context: it already fails in dimension 2, as  simple examples show. 
\end{rmk}

\section{Falling onto repelling cycles}\label{sec:main}

In this section we prove Theorem \ref{thm:main}. Near   $\lo\in \supp(T_1\wedge \cdots \wedge T_k)$, 
we want to to construct parameters at which the marked critical points fall onto repelling cycles. The 
strategy is to first make these critical points fall into a (persistent) non-polar  hyperbolic set, and then 
use the transversality results of the previous section to ``push" them towards repelling cycles.

\begin{proof}[Proof of Theorem \ref{thm:main}]
The first step is the following classical lemma. 

\begin{lem}\label{lem:hyperbolic}
Let $f$ be a rational map of degree $d\geq 2$ on $\pu$. There exists an integer $m$ and  a 
$f^m$-invariant compact set $K$ such that:
\begin{itemize}
\itm $f^m\rest{K}$ is uniformly hyperbolic and conjugated to a one-sided shift on two symbols;
\itm the unique balanced measure $\nu$ on $K$ (that is, such that $(f^m)^*\nu = 2\nu$) has (Hölder) 
continuous potential.
\end{itemize}
\end{lem}

\begin{proof}
It follows from the proof of the equidistribution of repelling points (see \cite{lyubich}) 
 that there exists an open ball $B$, an integer $m$ and two univalent  inverse branches $f_1^{-
m}$ and $f_2^{-m}$ of $f^m$ on $B$ such that $f_i^{-m}(B)\Subset B$  and $\overline{f_1^{-m}(B)}\cap 
\overline{f_2^{-m}(B)} = \emptyset$. Thus, letting $B_i = f_i^{-m}(B)$, we may consider 
$$K =\set{z\in  B_1\cup B_2, \ \forall n \in\nn,  \ f^{mn}(z)\in B}.$$ It is well-known that $f^m\rest{K}$ is 
uniformly hyperbolic and conjugate to a one-sided 2-shift. 

In addition, $K$ is not polar. More precisely it can be shown that the unique balanced 
measure under $f^n$ has  Hölder continuous  potential, see \cite[Theorem 3.7.1]{ds-pl} for a proof.  
\end{proof}

Let $(f_\la)$ be a family of rational mappings with $k$ marked critical points as in the statement of 
Theorem \ref{thm:main}, and let $\la_0$ be a parameter in  $\supp(T_1\wedge \cdots \wedge T_k)$.
We apply Lemma \ref{lem:hyperbolic} to $f_{\la_0}$, which provides us with a $f_\lo^m$-invariant 
hyperbolic set $K_\lo$, endowed  with a measure $\nu_\lo$.   Replacing  the family 
$(f_\la)$ by $(f^m_\la)$, it is no loss of generality to assume that $m=1$. Note that 
 $c_1, \ldots , c_k$ are still marked critical points, and the associated bifurcation currents are the same. 
Also, we rename $\lo$ into 0. 

The hyperbolic set $K_0$ persists in some neighborhood   of $0$. More precisely there exists a 
neighborhood $N$ of $0$,  biholomorphic to a ball,  and a  holomorphic motion  $(h_\la)_{\la\in N}$ of 
$K_\lo$ such that 
$h_\la \circ f_0 = f_\la\circ h_\la$. We set $K_\la = h_\la(K_0)$, and $\nu_\la = (h_\la)_*\nu_0$. Without 
loss of generality we rename $N$ into $\La$.

\medskip

In $\La\times \pu$,   consider the uniformly laminar current induced by the holomorphic motion of the measure $\nu_0$. 
More precisely, for $z\in K_0$, we let $  \widehat{h}\cdot z$ be the graph of the holomorphic motion 
over $\La$ passing through $z$, and let $\widehat \nu$ be the (1,1) positive closed current  defined by 
\begin{equation}\label{eq:nu}
 \widehat \nu = \int \big[ \widehat{h}\cdot z\big] d\nu_0(z). 
\end{equation}
 
\begin{lem}
$\widehat \nu$ has bounded potentials. 
\end{lem}

\begin{proof}
Since $K_0$ is contained in some affine chart $\cc\subset \pu$, 
reducing $\La$ again if necessary,  we may view   
 $\widehat \nu$  as a  {\em horizontal current} in $\La\times \cc$, in the sense that $\supp \widehat
\nu \cap (\set{\la}\times \cc)$ is compact for every $\la\in\La$.    It is classical that 
  a potential of  $\widehat\nu$ is then given by 
 $$u(\la, z) = \int \log\abs{s-z} d\nu_\la(s). $$ A proof of this fact is given  in dimension 2 in \cite[\S 6]
{structure}, which easily adapts to higher dimensions.   
  The Hölder continuity of holomorphic motions implies that, restricted to $\supp(\nu_\la)$,  $u(\la, 
\cdot)$ is Hölder continuous  as a function of $z$, locally uniformly in $\la$. In particular   $u$ is locally 
uniformly bounded on $\supp\big(\widehat\nu \big)$, hence everywhere, by the maximum principle. 
The result follows. Adapting \cite[\S 6]{structure}, it is not difficult to see that $u$ is actually locally 
Hölder continuous.
 \end{proof}

\medskip

Let us now work in $\La\times (\pu)^k$ and, similarly to Theorem \ref{thm:value}, 
investigate the motion of the $k$ marked critical points 
relative to $K_\la$.  For $1\leq j \leq k$   let   
$\widehat \nu_j = \widehat p _j^* \widehat \nu$, where $\widehat p _j:\La\times (\pu)^k\cv \La\times \pu$    is defined 
by  $(\la, z_1,\cdots, z_k)\mapsto (\la, z_j)$, and $\widehat \nu$ is as in \eqref{eq:nu}.    
We put $V = (k!)^{-1/k}\lrpar{\widehat \nu_1 + \cdots + \widehat \nu_k}$ so that $V^k = \widehat \nu_1 
\wedge  \cdots \wedge  \widehat \nu_k$, and  letting $\widehat C_n (\la):= (\la, C_n(\la))$,
we study the sequences of currents 
$\widehat C_n^* V$ and $\widehat C_n^* V^k$. 

Note that we are abusing slightly  here since there  does not exist a well-defined pull back operator $\widehat C_n^*$ 
from currents on $\La\times (\pu)^k$ to currents on $\La$, even for positive closed currents of 
bidegree (1,1), since $\widehat C_n$ is obviously not dominant. 
 The meaning of this notation here is the following: $V$ is a positive closed current of bidegree (1,1) on 
$\La\times (\pu)^k$ with bounded potentials, so the wedge product $V\wedge [\Gamma(C_n)]$ of $V$ 
with the graph of $C_n$ is well defined, and corresponds to the restriction of $V$ to $\Gamma(C_n)$. On the other hand $\pi_\La\rest{\Gamma(C_n)} : \Gamma
(C_n)\cv \La$ is a biholomorphism, so we can set $\widehat C_n^* V = (\pi_\La)_* \lrpar{V \wedge 
[\Gamma(C_n)]}$.  Writing locally $V=dd^c_{(\la,z)} (v(\la, z))$, we obtain   that  $\widehat 
C_n^* V = dd^c _\la (v(\la, C_n(\la)))$ --recall that a psh function has a value at every point. 

 The  definition of $\widehat C_n^* V^k$ is similar, arguing by induction, since for every $1\leq j\leq k$,   $V^j \wedge 
[\Gamma(C_n)]$ is well-defined. These definitions coincide of course with the usual ones for smooth 
forms, and are stable under regularizations.  
 
 \medskip
 
 This being said, it is easy that $d^{-n}\widehat C_n^* V$ converges to $(k!)^{-1/k}(T_1+ \cdots + T_k)$, where in addition  the potentials converge uniformly. Hence $d^{-n}\widehat C_n^* V^k$ converges to 
 $T_1\wedge \cdots \wedge T_k$. 
 
 Indeed, let $\widehat \kappa$ be the pull back of the previously defined K\"ahler form $\kappa$ under the natural projection $\La\times (\pu)^k\cv (\pu)^k$. 
 We know from the proof of  Theorem \ref{thm:value} that the  sequence $d^{-n} \widehat C_n^*\widehat \kappa = 
 d^{-n}C_n^*\kappa$ converges 
 to $(k!)^{-1/k}(T_1+ \cdots + T_k)$, with uniform convergence of the potentials.
Now  we simply write  $V - \widehat \kappa = dd^c w $, where $w$ is a uniformly bounded function on $\La\times (\pu)^k $, so 
$$d^{-n}\widehat C_n^* V = d^{-n} \widehat C_n^*\kappa + d^{-n} dd^c(w\circ \widehat C_n),$$  
and the result follows.
 
 \medskip
 
We now give  a geometric interpretation of the current $V^k$  and    its intersection with $\Gamma(C_n)$. For $ (z_1, \ldots , z_k)\in K_0^k$, we let 
 $$\widehat h\cdot (z_1, \ldots , z_k) = \set{(\la, h_\la(z_1), \ldots , h_\la(z_k)), \ \la\in \La}$$ be the graph of its continuation under the product holomorphic motion. 
The first observation is that the current $V^k$ is an integral of graphs over $\La$: 
\begin{equation}\label{eq:Vk}
V^k = \widehat\nu_1\wedge \cdots\wedge  \widehat \nu_k = \int \big[\widehat h\cdot (z_1, \ldots , z_k)\big] d\nu(z_1)\cdots d\nu(z_k).
\end{equation}
Indeed, it is clear that for every $j$, $\widehat{\nu}_j$ is a uniformly laminar current, since it is the pullback under $\widehat{p}_j$ of the uniformly laminar current $\widehat \nu$. 
 The geometric interpretation  \eqref{eq:Vk} of $V^k$ then follows from Theorem \ref{thm:isect geom}.  Applying now the same theorem  to $V^k\wedge [\Gamma(C_n)]$, we obtain that 
\begin{equation}\label{eq:Vk Gamma}
 V^k \wedge [\Gamma(C_n)] = \int \big[\widehat h\cdot (z_1, \ldots , z_k) \wedge \Gamma(C_n) \big] d\nu(z_1)\cdots d\nu(z_k),
\end{equation} and furthermore\footnote{Notice that while the  transversality of the leaves of the $\widehat\nu_j$ is obvious, this is not anymore the case after restriction to $\Gamma(C_n)$. This is where we need the full strength of Theorem \ref{thm:isect geom}. } that the intersection 
chain $\widehat h\cdot (z_1, \ldots , z_k) \wedge \Gamma(C_n)$ is generically of multiplicity 1. Recall that this means that  the intersection $\widehat h\cdot (z_1, \ldots , z_k) \cap \Gamma(C_n)$ is geometrically transverse along  $\mathrm{Reg}(\widehat h\cdot (z_1, \ldots , z_k) \wedge \Gamma(C_n))$, which is of full trace measure in $[\widehat h\cdot (z_1, \ldots , z_k) \wedge \Gamma(C_n)]$.

\medskip

We are now in position to conclude the proof.  Since 
$d^{-n}\widehat{C}_n^* V^k = (\pi_\La)_* \big(V^k \wedge [\Gamma(C_n)] \big)$ converges to $T_1\wedge \cdots \wedge T_k$  and  $\la_0 =0 \in \supp(T_1\wedge \cdots \wedge T_k)$, then by 
\eqref{eq:Vk Gamma},  there exists a sequence of parameters $\la_n$ converging to $0$ 
such that $\Gamma(C_n)$ and some graph $\widehat h\cdot 
(z_1^n, \ldots , z_k^n)$ of the   holomorphic motion 
 intersect transversely  over $\la_n$. Repelling periodic orbits for $f_0$ are dense in 
$K_0$, so for fixed (large) $n$, for $j=1,  \ldots , k$,  there exists a sequence   
of repelling $f_0$-periodic points 
$(z_j^{n,q})_{q\geq 1}$ belonging to  $K_0$, and  converging to  $z_j^{n}$. 
The continuations $h_\la(z_j^{n,q})$ remain repelling throughout $\La$ by hyperbolicity.
By the persistence of transverse intersections, for large $q$,
$\widehat h\cdot  (z_1^{n, q}, \ldots , z_k^{n, q})$ intersects $\Gamma(C_n)$ transversely 
near $\la_n$. Thus we have found parameters close to $0$ at which $c_1, \ldots, c_k$ fall transversely onto repelling cycles. 
\end{proof}

\section{When $\La$ is the space of polynomial or rational maps}\label{sec:polynomial}

Here we show that in the case where $\La$ is the space of all polynomials or rational maps with all critical points marked, the statement of Corollary \ref{cor:support} takes a  simpler form.  
Similar ideas already   appeared  in the polynomial case in \cite[\S 8]{gauthier}.

\medskip

 We say that a family $(f_\la)$ of rational maps is {\em reduced} if it is generically transverse to the orbits of the $\PSL$-action by conjugacy on the space of rational maps of degree $d$. In other words, we require that for every $\lo\in \La$, the set of parameters $\la\in \La$ such that $f_\la$ is holomorphically conjugate to $f_\lo$ is discrete. 

We start with a general result.
 
\begin{prop}\label{prop:proper}
Let $(f_\la)_{\la\in\La}$ be a reduced  algebraic family of rational maps of degree $d$, with marked critical points $c_1,\ldots , c_\ell$. Assume that for all  $2\ell$-tuples  of integers $(n_j\geq 0, \ m_j\geq 1)_{1\leq j\leq \ell}$, the subvariety defined by $\bigcap_{j=1}^\ell \set{\la, f_\la^{n_j}(c_j(\la)) = f_\la^{n_j + m_j}(c_j(\la))}$, whenever non-empty,  is of pure codimension $\ell$. 

Then for every $k\leq \ell$, 
$$\supp(T_1\wedge \cdots \wedge T_k) = \overline{\set{ \la, c_1(\la), \ldots , c_k(\la) \text{ fall onto repelling cycles}}}$$
\end{prop}

We use  a result of Gauthier \cite[Thm 6.2]{gauthier}, that we briefly describe now.
Assume that for some $\lo\in \La$, the marked  critical points $c_j$ eventually land on repelling periodic points $p_j$, that is $f^{n_j}_\lo(c_j(\lo)) = p_j(\lo)$. As in the introduction,
let   $C$ and $P$ be respectively defined by $C:\la\mapsto 
(f^{n_1}_\la(c_1(\la)), \ldots, f^{n_k}_\la(c_k(\la))$ and $P:\la\mapsto (p_1(\la), \ldots , p_k(\la))$, where  $p_j(\la)$ is the natural continuation of $p_j(\lo)$. 
We say that  the critical points $c_j$ {\em fall properly} onto the  respective repelling points  $p_j$ at $\lo$  if the     graphs of the two mappings 
$\la\mapsto (p_1(\la), \ldots p_k(\la))$ and $\la\mapsto (f^{n_1}_\la(c_1(\la)), \ldots , 
f^{n_k}_\la(c_k(\la)))$ intersect properly at $(\la_0, p_1(\la_0), \ldots p_k(\la_0))$. Denoting by $m_j$ the period of   $p_j$, we see that for this it is enough that the subvariety 
$\bigcap_{j=1}^k\set{\la, f_\la^{n_j}(c_j(\la)) = f_\la^{n_j + m_j}(c_j(\la))}$ has codimension $k$ at $\la_0$. Gauthier's theorem  asserts that if the critical points $c_j$ {fall properly} onto the  respective repelling points  $p_j$ at $\lo$, then $\lo\in \supp(T_1\wedge \cdots \wedge T_k)$. 

\begin{proof}
In view of Theorem \ref{thm:main}, we only need to show that if the 
 $c_j(\lo)$, $1\leq j\leq k$ eventually land on repelling periodic points, then 
 $\lo\in \supp(T_1\wedge \cdots \wedge T_k)$.
For this,  in order to use the above result, we 
 show that under the assumptions of the proposition, for every $k\leq \ell$, and all integers $(n_j\geq 0, \ m_j\geq 1)_{1\leq j\leq k}$, 
$\bigcap_{j=1}^k \set{\la, f_\la^{n_j}(c_j(\la)) = f_\la^{n_j + m_j}(c_j(\la))}$, whenever non-empty,  
is of pure codimension $k$ (notice that this does not follow directly from the   assumption of the proposition<, see \cite[pp. 143-144]{chirka}). The proof is by decreasing induction on $k$, so assume that the result holds at step $k+1$, and let $W$ be an irreducible component 
of $\bigcap_{j=1}^k \set{\la, f_\la^{n_j}(c_j(\la)) = f_\la^{n_j + m_j}(c_j(\la))}$ for arbitrary integers 
$(n_j\geq 0,m_j\geq 1)_{1\leq j\leq k}$. A first possibility is that  $c_{k+1}$ is passive on $W$.
Since $W$ is algebraic,  \cite[Thm 2.5]{preper} asserts that either all mappings on $W$ are holomorphically conjugate, or $c_{k+1}$ is persistently preperiodic on $W$. The first case is excluded because $(f_\la)$ is reduced, and, since $\codim(W)\leq k$,  the second case contradicts the 
  induction hypothesis. We then infer   that the activity locus of $c_{k+1}$ along $W$ is non-empty. Therefore, there are non-empty proper hypersurfaces in $W$ where $c_{k+1}$ becomes preperiodic. By the induction  hypothesis the codimension of such a hypersurface in $\La$ equals 
$k+1$, hence $\codim(W) =k$.
\end{proof}

Let $\mathcal{P}_d^{\mathrm cm}$ be the space of polynomials of degree $d$ with marked critical points, up to affine conjugacy. This space is described in detail in \cite{preper}: it is an affine algebraic variety of dimension $d-1$, which admits a finite branched cover by $\cc^{d-1}$. Denote by $(c_j)_{j=1,\ldots ,d-1}$ the marked critical points and by $T_j$ the associated bifurcation currents. 
 Applying Proposition \ref{prop:proper} we thus recover the following result from \cite[\S 8]{gauthier}.
 
\begin{cor}
In $\mathcal{P}_d^{\mathrm cm}$, for every $k\leq d-1$, we have that 
$$\supp(T_1\wedge \cdots \wedge T_k) = \overline{\set{ \la, c_1(\la), \cdots , c_k(\la) \text{ fall onto repelling cycles}}}$$
\end{cor}

\begin{proof}
The assumption of  Proposition \ref{prop:proper} is satisfied for  $\ell = d-1$. Indeed, for all $(n_j\geq 0, \ m_j\geq 1)_{1\leq j\leq d-1}$, $\bigcap_{j=1}^{d-1} \set{\la, f_\la^{n_j}(c_j(\la)) = f_\la^{n_j + m_j}(c_j(\la))}$ is of dimension 0, because it is  contained in the connectedness locus, which is compact in $\cc^{d-1}$ by Branner-Hubbard \cite{branner-hubbard1}.
\end{proof}

Let now $\mathcal{M}_d^{cm}$  be the space of rational mappings of degree $d$, up to Möbius conjugacy, with marked critical points $ (c_j)_{j=1,\ldots ,2d-2}$. It is a normal quasiprojective variety of dimension $2d-2$ \cite{silverman}.  Proposition \ref{prop:proper} then leads to the following:

\begin{cor}
In $\mathcal{M}_d^{\mathrm cm}$, for every $k\leq 2d-2$, we have that 
$$\supp(T_1\wedge \cdots \wedge T_k) = \overline{\set{ \la, c_1(\la), \ldots , c_k(\la) \text{ fall onto repelling cycles}}}$$
\end{cor}

\begin{proof}
A difficulty here is that the assumption of   proposition \ref{prop:proper} is not   valid in general 
for $\ell = 2d-2$ due to the possibility of {\em flexible Lattès examples} (see Milnor \cite{Milnor lattes} for a general account on Lattès examples). 
So we treat the cases $k\leq 2d-3$ and $k=2d-2$ separately. For $k=2d-2$, the result was proven by 
 Buff and Gauthier \cite{buff gauthier}. 

 Let us show    that the assumption of Proposition \ref{prop:proper} holds for $k=2d-3$. For this,
 assume that for some $(n_j, m_j)_{1\leq j\leq 2d-1}$ as above, 
$\bigcap_{j=1}^{2d-3} \set{\la, f_\la^{n_j}(c_j(\la)) = f_\la^{n_j + m_j}(c_j(\la))}$ is not empty, and let us prove that it is of pure codimension $2d-3$ (i.e. of pure dimension 1). If not,   it admits a component $W$ of dimension  greater than 1. If the last free critical point $c_{2d-2}$ is passive on $W$, 
then the family  $(f_\la)$ is stable along $W$. It then  follows from a theorem of McMullen \cite{mcmullen algorithms} 
that $W$ is a (reduced) family of flexible Lattès examples, which is impossible because such a family should have 
 dimension 1. Therefore, the activity locus of  $c_{2d-2}$ is non-empty, giving rise  to  algebraic hypersurfaces 
$H_m\subset W$ such that 
where $c_{2d-2}$ is  persistently preperiodic to periodic points of arbitrary large  period $m$ (of course the minimal 
period may drop to some divisor of $m$, but this only happens on  a proper subvariety of $H_m$). As before, $H_m$ must be a family of flexible Lattès examples, which is impossible for in this case  the critical points must eventually fall on   repelling points of period 1 or 2 (see  \cite{Milnor lattes}).
This contradiction finishes the proof.
\end{proof}

\end{document}